\documentclass{amsart}

\usepackage{amssymb,amsfonts,amsxtra,amsmath,tikz-cd,xcolor,mathrsfs,verbatim,centernot,enumitem,mathtools}
\IfFileExists{quiver.sty}{\usepackage{quiver}}{}
\usepackage{relsize}    
\usepackage[all]{xy}
\IfFileExists{dsfont.sty}{\usepackage{dsfont}}{\let\mathds\mathbb}
\usetikzlibrary{decorations.markings}
\tikzset{negated/.style={
        decoration={markings,
            mark= at position 0.5 with {
                \node[transform shape] (tempnode) {$\backslash$};
            }
        },
        postaction={decorate}
    }
} 
\makeatletter
\def\@firstoftwo@second#1#2{%
  \def\temp##1.##2\@nil{##2}%
   \temp#1\@nil}
\newcommand\sref[1]{%
   \expandafter\@setref\csname r@#1\endcsname\@firstoftwo@second{#1}%
}

\makeatother

\newcommand{\ind}{\mathds{1}} 

\newtheorem{theo}{Theorem}[section]
\newtheorem*{theo*}{Theorem} % Unnumbered version
\newtheorem{lemm}[theo]{Lemma}

\newtheorem{conj}[theo]{Conjecture}
\newtheorem*{conj*}{Conjecture} % Unnumbered
\newtheorem{coro}[theo]{Corollary}

\newtheorem{conv}[theo]{Convention}
\newtheorem{prop}[theo]{Proposition}

\theoremstyle{definition} 

\newtheorem{exam}[theo]{Example}
\newtheorem*{exam*}{Example} % Unnumbered
\newtheorem{defi}[theo]{Definition}

\theoremstyle{remark}
\newtheorem{rema}[theo]{Remark}

\numberwithin{equation}{section}

\newcommand{\CC}{\mathbb{C}}

\newcommand{\ZZ}{\mathbb{Z}}

\newcommand{\cI}{\mathcal{I}}
\newcommand{\cO}{\mathcal{O}}

\newcommand{\cL}{\mathcal{L}}

\newcommand{\cT}{\mathcal{T}}

\newcommand{\cP}{\mathcal{P}}

\newcommand{\cJ}{\mathcal{J}}

\newcommand{\Spec}{\operatorname{Spec}}

\newcommand{\rk}{\operatorname{rk}}
\newcommand{\tr}{\operatorname{trace}}

\def\and{\quad\mathrm{and}\quad}

\def\Eu{\textup{Eu}}
\def\Der{{\textup{Der}}}
\def\PH{{\textup{PH}}} 

\def\GSV{{\textup{GSV}}}

\def\Ind{{\textup{Ind}}}
\def\Sing{{\textup{Sing}}}

\def\beq{\begin{equation}}
\def\eeq{\end{equation}}

\newcommand{\fm}{\mathfrak{m}}
\DeclareMathOperator{\Soc}{Soc}
\DeclareMathOperator{\Res}{Res}
\DeclareMathOperator{\Hess}{Hess}

\usepackage[hidelinks]{hyperref}

\begin{document}
\title[Criteria for Weighted Homogeneity via Logarithmic Vector Fields]{Criteria for Weighted Homogeneity via Logarithmic Vector Fields}

\author{Jihao Liu}
\address{Department of Mathematics, Peking University, No. 5 Yiheyuan Road, Haidian District, Beijing 100871, China}
\address{Beijing International Center for Mathematical Research, Peking University, No. 5 Yiheyuan Road, Haidian District, Beijing 100871, China}
\email{liujihao@math.pku.edu.cn}

\author{Xiping Zhang}
\address{School of Mathematical Sciences, Key Laboratory of Intelligent Computing and Applications (Ministry of Education),
Tongji University, 
1239 Siping Road, 
Shanghai, China}
\email{xzhmath@gmail.com}
%\thanks{The   author  is supported by National Natural Science Foundation of China (Grant No.12201463). }

%    General info
%\subjclass[2020]{14B05, 14C17}

%\date{January 1, 2001 and, in revised form, June 22, 2001.}


%\keywords{ }

\begin{abstract} 
Recently in \cite{Machado-Seade26} the authors proved a lower bound for the $\GSV$ index of a logarithmic vector field along an isolated hypersurface singularity 
and 
proposed a conjecture that the weighted homogeneity of this hypersurface singularity  can be detected by the existence of either a transverse or a non-degenerate  holomorphic logarithmic vector fields. 

In this paper we prove this conjecture affirmatively. 
We also prove that the $\GSV$ index attains the lower bound if and only if the hypersurface singularity is weighted homogeneous. 
\end{abstract}
\maketitle

\section{Introduction} 
The study of 
Singular varieties  have gained significant importance in modern algebraic geometry. 
One of the most fundamental prototypes of  singularities is the isolated weighted homogeneous hypersurface singularities. 
Such singularities are relatively simple, yet they contain rich geometric information, and they have been studied extensively across various branches of mathematics.

Let $(W,0)\subset\CC^{n+1}$ be an open domain and let $D\subset W$ be the germ of a reduced hypersurface with an isolated singularity at $0$ cut by a holomorphic function germ  $f\in \cO_{W,0}$.  
Recall that  $(D,0)$ is called   weighted homogeneous  if,  in some local analytic coordinates the defining equation of $D$ is weighted homogeneous.      
By Saito’s results in  \cite{KSaito71}, 
this is equivalent to  $f$ being quasi-homogeneous, i.e., the Jacobian ideal $\cJ_f=(f_{x_0}, f_{x_1}, \cdots , f_{x_n})$ contains $f$. When $D$ has an isolated singularity at $0$, this is equivalent  to $\tau=\mu$, where $\tau:=\dim_\CC \frac{\cO_{W,0}}{(f)+\cJ_f}$ and 
$\mu:=\dim_\CC \frac{\cO_{W,0}}{\cJ_f}$ denote the Tjurina number and the Milnor number of $D$ at $0$.
 
We say a holomorphic vector field $\tilde{\nu}\in \Der_{W,0}$ is logarithmic along $D$ if 
$\tilde{\nu}$ is tangent to $D$ at its smooth points.
Recently, da Silva Machado and Seade studied lower bounds for the
$\GSV$ indices $\GSV(\tilde{\nu},D,0)$ of such holomorphic vector fields in \cite{Machado-Seade26}. In \cite[Theorem~1]{Machado-Seade26} they proved that
\begin{theo}
\label{theo; MSbound}
Every logarithmic holomorphic vector field
$\tilde{\nu}$ with an isolated singularity
at $0$ satisfies
\[
\GSV(\tilde{\nu},D,0)\geq 1+(-1)^n\tau.
\]
When $n=1$, we need to assume additionally that $D$ is irreducible.
\end{theo}

In \cite{Machado-Seade26} da Silva Machado and Seade alsp proposed the following conjecture. 
\begin{conj}
\label{c1}
Let $D$ be the isolated hypersurface singularity as above.
The following statements are equivalent.
\begin{enumerate}
\item[$(i)$] $(D,0)$ is  weighted homogeneous.
\item[$(ii)$] There exists a holomorphic vector field tangent to $D$ with an isolated singularity at $0$ and 
is transverse to  the link spheres of $D$. 
\item[$(iii)$] There exists a holomorphic vector field 	tangent to $D$ with a non-degenerate isolated singularity at $0$.
\end{enumerate}
\end{conj} 
The implication $(i)\implies (ii)$ and $(i)\implies (iii)$ are obvious, where one just takes the weighted Euler vector field.  
In  \cite{Machado-Seade26} the authors proved that
$(ii)\implies (i)$ when $n$ is odd and $(iii) \implies (i)$ when $n$ is even using  \cite[Theorem 1]{Machado-Seade26}. 

In the first part of this paper we discuss this conjecture. We  adapt the microlocal intersection formula from \cite{LZ25-2} and prove this conjecture completely.
\begin{theo}
\label{t1} 
Conjecture~\ref{c1} is true.
\end{theo} 
The proof is summarized as follows.  For the implication $(iii)\implies (i)$ we adapt the microlocal intersection formula for the $\GSV$ index given in \cite{LZ25-2}, together with the fact that the local Euler obstruction of a
non-degenerate holomorphic vector field agrees with the classical local Euler
obstruction (see Theorem~\ref{theo; non-degenerate}).  For the implication $(ii)\implies (i)$, notice that the transverse holomorphic vector field produces a holomorphic  contracting flow, integrating along which we obtain
  a contracting automorphism of the germ $(D,0)$. It then follows from \cite[Theorem A]{Morvan24} that $(D,0)$ is weighted
homogeneous (see Corollary~\ref{coro; conjsolved}). In fact, by Remark~\ref{rema} a transverse  vector field is necessarily non-degenerate, thus we may also adapt implication $(iii)\implies (i)$ to conclude  $(ii)\implies (i)$.

In the second part of this paper we characterize 
 when the minimal value $1+(-1)^n \tau$ in  Theorem~\ref{theo; MSbound} is attained.
When $D$ is weighted homogeneous, the $\GSV $ index of the Euler vector field always achieves the minimal value. Conversely, in Corollary~\ref{coro; GSVlowerbound} we prove that this is the only possible case.
  This gives a further criterion
for weighted homogeneity, supplementing   Conjecture~\ref{c1}.
\begin{theo}
\label{theo; GSVlowerboundintro}
The hypersurface singularity $(D,0)$ is weighted homogeneous if and
only if there exists a  holomorphic vector field $\tilde{\nu}$ logarithmic along $D$ and with an isolated singularity at $0$, such that $\GSV(\tilde{\nu},D,0)=1+(-1)^n\tau$ attains the minimal value.
\end{theo}

For odd $n$ this is an immediate consequence of the algebraic formula for the $\GSV $ index from \cite{MR1642757} and Conjecture~\ref{c1} (see also \cite[Lemma 1.1-(c)]{Machado-Seade26}.
The proof for even $n$ is however more involved.
First we modify Suwa's  residue formula  for the $\GSV$ index (\cite{Suw96})
to show that the Jacobian matrix of $\tilde{\nu}$ has rank $n$ (see Proposition~\ref{prop; residueGSV} and Proposition~\ref{prop; ranklinearpart}).
In Lemma~\ref{lemm; Hessianrank} we prove that the rank of  the Hessian of $f$ is bounded below by the Jacobian matrix of $\tilde{\nu}$, hence it has rank $\geq n$.  
Then using analytic splitting lemma we derive that  
$f=y_1^2+\cdots +y_n^2+z^m$ for some local coordinates $\{z, y_1, \cdots ,y_n\}$ and $m\geq 2$.   Clearly $f$ is weighted homogeneous.

\vskip .1in
\noindent
{\bf Acknowledgements}: 
The  authors would like to thank Xia Liao for carefully reading the first draft and providing valuable feedbacks.
The  authors  thank  Yuanpu Xiong for helpful  discussions.  The first author was partially supported by the National Key R\&D Program of China \#\allowbreak 2024YFA1014400. 
The second  author was supported by National Natural Science Foundation of China (Grant No.12201463).  

\vskip .1in
\noindent
{\bf AI Disclosure}
We   note that, parallel to this paper, in \cite{LiuZ26} a different proof of Conjecture~\ref{c1}  is given   by
  the Danus system, a specialized agent built on Rethlas and substantially more capable for fundamental mathematical research based on the Rethlas system (see \cite{Ju+26} for an introduction to the Rethlas system).  While \cite{LiuZ26} was generated entirely by AI, this presented draft is written completely  by human, while  AI was  only used  on literature searching and   explaining the relevant references.
 
\section{Preliminary}
\subsection{Vector Fields} 
Let $(W,0)\subset\CC^{n+1}$ be an open domain and let $D\subset W$ be the germ of a reduced hypersurface  cut by a holomorphic function germ  $f\in \cO_{W,0}$.     
Let $\Der_{W,0}$ be the space of germs of holomorphic vector fields on $W$. 
We say a holomorphic vector field $\tilde{\nu}=\sum_{i=0}^n \tilde{\nu}_i \partial_{x_i}\in \Der_{W,0}$  is tangent to $D$ if the restriction $\tilde{\nu}|_{D_{sm}}$ is tangent to the smooth locus $D_{sm}$, i.e.,  $\tilde{\nu}(P)\in T_PD_{sm}$ for every $P\in D_{sm}$. This is equivalent to saying that
\[
\tilde{\nu}\in \Der_{W,0}(-\log D):=
\left\lbrace
\chi=\sum_{i=0}^n \chi_i\partial_{x_i}\in \Der_{W,0}
\  \Big\vert \  \chi(f)=\sum_{i=0}^n \chi_i\cdot f_{x_i} \in (f) 
\right\rbrace \/.
\] 
Such holomorphic vector fields  will be called logarithmic holomorphic vector fields (along $D$).

\begin{conv}
\label{conv}
In this paper  we always assume that $D$ is reduced and has an isolated singularity at $0$. If $n=1$ we further assume that $D$ is irreducible.
 We will denote by $\mu$ and $\tau$   the   Milnor number  and the  Tjurina number of $D$ at $0$ respectively.

We also assume that  $\tilde{\nu}=\sum_{i=0}^n \tilde{\nu}_i \partial_{x_i}\in \Der_{W,0}(-\log D)$ is a logarithmic holomorphic vector field  along $D$ and that $\tilde{\nu}$ has an isolated singularity at $0$, i.e.,
$
\Sing(\tilde{\nu})
:= \operatorname{Supp} V(\tilde{\nu}_0,\tilde{\nu}_1,\ldots,\tilde{\nu}_n)
= \{0\}.
$

We equip $W$ with the standard Hermitian metric, and hence with the standard Euclidean metric inherited from $\mathbb{C}^{n+1}$. With respect to this metric, we denote by $B_\epsilon(0)$ and $S_\epsilon$ the $(2n+2)$-dimensional ball and the $(2n+1)$-dimensional sphere centered at $0$ with radius $\epsilon$, respectively.
\end{conv}

\begin{defi}
\label{defi; transverse} 
We say that a logarithmic holomorphic vector field $\tilde{\nu}$ is transverse to the
link spheres of $D$ if there exists $\epsilon_0>0$ such that, for every
$0<\epsilon<\epsilon_0$ and every point $P\in D\cap S_\epsilon$, one has $\text{Re}(\tilde{\nu}(P))\notin T_PS_\epsilon $. 
This is equivalent to saying that 
$dr(\text{Re}(\tilde{\nu}))(P)    \neq 0$ for any point $P\in D\cap S_\epsilon$, where  $r=\|x\|^2$ denotes the real   distance square function.
\end{defi}

\begin{defi}
\label{defi; non-degenerate}
We say that a logarithmic holomorphic vector field  $\tilde{\nu}$ has a non-degenerate isolated singularity at $0$  if $\Sing(\tilde{\nu})=\{0\}$  and the linear part of $\tilde{\nu}$ is non-degenerate. This is equivalent to saying that the Jacobian ideal
$\cI_{\tilde{\nu}}:=(\tilde{\nu}_0, \tilde{\nu}_1, \cdots , \tilde{\nu}_n)$  coincides with the maximal ideal $\mathfrak{m}_{W,0}$ in $\cO_{W,0}$. 
\end{defi}

\begin{exam}
\label{counterexam}
Set $f=x^2+y^2+z^2\in \mathbb C\{x,y,z\}$ and  $D=\{f=0\}\subset \CC^3$. We consider the Euler vector field  $E=x\partial_x+y\partial_y+z\partial_z$ and the vector field $\tilde{\nu}
:=y\partial_x-x\partial_y+f\partial_z$. Then we have $E(f)=2f$ and $\tilde{\nu}(f)=2zf$.
Both vector fields are then tangent to $D$ and have isolated singularities at $0$. 

The Euler vector field $E$ is    transverse to the
link spheres of $D$ and has a non-degenerate singularity at $0$. The vector field  $\tilde{\nu}$ however, does not have a non-degenerate singularity at $0$ since its linear part is $y\partial_x-x\partial_y$, which is degenerate. Also by computation we have $dr(\text{Re}(\tilde{\nu}))(P)   \equiv 0$ for any point $P\in D\cap S_\epsilon$, thus this vector field is everywhere tangent to the link spheres of $D$ instead of being transverse to the link spheres.

On the other hand,  computing the  $\GSV$ index  (see Proposition~\ref{prop; algebraicGSV}) we obtain   
\[
\GSV(\tilde{\nu}, D, 0)=
\dim_\CC \frac{\cO_{W,0}}{(y, -x, f)} -\dim_\CC \frac{\cO_{W,0}}{(y, -x, f, 2z)}+\tau
=2-1+\tau =1+\tau \/.
\]  
This is  the lower bound as in \cite[Theorem 1]{Machado-Seade26}. 
However   $\nu:=\tilde{\nu}|_D=y\partial_x-x\partial_y$ does not admit any non-degenerate holomorphic extensions, since any extension $\chi$ of $\nu$ differs $\tilde{\nu}$ by a vector field $\eta\in (f)\Der_{W,0}$, and therefore cannot have non-degenerate linear part.  
\end{exam}

\subsection{Characteristic Cycles and Microlocal Intersection Formula}
Given the    open domain $W\subset \CC^{n+1}$, for any irreducible closed  subvariety $X\subset W$ we consider its conormal space
\[
T_X^*W:=\overline{
\left\lbrace (x, \xi)\vert x\in X_{sm} \text{ and } \xi(T_xX) =0 
\right\rbrace } \subset T^*W \/.
\]
This is an irreducible conical Lagrangian subvariety of $T^*W$. In fact one can prove that any irreducible conical Lagrangian subvariety in $T^*W$ is a conormal space. We denote by $\cL(W)$  the free abelian group generated by irreducible  conical Lagarangian subvarieties. 

By a constructible function on $W$ we mean an integer valued function $\gamma\colon W\to \mathbb{Z}$ such that $\gamma^{-1}(i)$ is a constructible subset of $W$. The constructible functions on $W$ form an abelian group denoted by  $CF(W)$.  
To generalize the Poincar\'e-Hopf theorem to singular varieties, in \cite{MAC} MacPherson introduced the local Euler obstruction functions $\{\textup{Eu}_X\}$ for  closed subvarieties $X$ of $W$. The local Euler obstruction functions  form 
 a  $\ZZ$-linear  basis for $CF(W)$.   More details about the local Euler obstructions will be given in \S\ref{sec; Euobs}.
 
\begin{defi}
\label{defi; CC}
The characteristic cycle map is the  group homomorphism
\[
\textup{CC}\colon CF(W) \to \cL(W) \/; \quad \Eu^\vee_X:=(-1)^{\dim X}\Eu_X  \mapsto   [T_X^*W] \/.
\]
For any constructible function $\gamma$, its image $\textup{CC}(\gamma)$ is called the characteristic cycle  of $\gamma$. 
\end{defi}

\subsection{Local Euler Obstructions of Holomorphic Vector Fields}
\label{sec; Euobs}
Let $W$ be an open domain of $\CC^{n+1}$ and let $X$ be a purely $d$-dimensional analytic subvariety of $W$.  
By the Nash blowup of $X$ (in $W$) we mean the following analytic variety
\[
Z:= \overline{ \left\lbrace
(x, T_xX) \vert x\in X_{sm}   \text{ and  } T_xX  \text{ is the tangent space} 
\right\rbrace} \subset \textup{Gr}_d(TW) \cong W\times \textup{Gr}(d,\CC^{n+1}) \/.
\] 
The first projection $\pi\colon Z\to X$ is a proper and  birational map. 
The universal subbundle of $\textup{Gr}_d(TW)$ restricts to a rank $d$ vector bundle $\cT$ on $Z$ and we call it the Nash tangent bundle. 

We assume that $X$ has an isolated singularity at $0$. 
Let $\tilde{\nu}=\sum_{i=0}^n \tilde{\nu}_i \partial_{x_i}$ be a holomorphic vector field on $W$ that is tangent to  $X_{sm}$ and we assume $\Sing(\tilde{\nu})=\{0\}$.
Then the restricted vector field $\nu:=\tilde{\nu}|_X$
naturally lifts to a holomorphic section $\chi$ of the Nash tangent bundle $\cT$ that  only vanishes  at $Z_0:=\pi^{-1}(0)$ in a neighborhood $V_\epsilon:=\pi^{-1}(X\cap B_\epsilon(0))$
for some $\epsilon > 0$ small enough. 

The  local Euler obstruction of $\tilde{\nu}$ at $0$, denoted by $\textup{Eu}_X^{\tilde{\nu}}(0)$, is then
 defined to be the obstruction to extending $\chi$ as a nowhere zero section of $\cT$ from $\partial V_\epsilon:=\pi^{-1}(X\cap S_\epsilon)$ to $V_\epsilon$.  
When $\tilde{\nu}$ is a radial vector field this coincides with the classical local Euler obstruction defined in \cite{MAC}.
For more details about obstruction theory and local Euler obstructions we refer to \cite{MR1835697}\cite{MR2759085}\cite{MR4476666}\cite{MR440554}.

By \cite[Proposition 3.2]{LZ25-2} we have  the following Gonz\'{a}lez-Sprinberg   formula.
\begin{equation} 
\label{eq; EulerObs} 
\textup{Eu}_{X}^{\tilde{\nu}}(0)=\int_{Z_0}  [Z]\cdot [\chi(Z)]=\int_{Z_0} c(\cT)\cap s(\tilde{Z}_0, Z)  \/.
\end{equation}
Here $\tilde{Z}_0$ is the intersection scheme of $\chi(Z)$ with the zero section of $\cT$ and $\int_{Z_0}: H_0(Z_0) \to \mathbb{Z}$ is the pushdown morphism induced by $Z_0\to pt$.
 
On the conormal side, the Hermitian metric  on $W$ induces an $\mathbb{R}$-linear $C^\infty$ isomorphism $TW \cong T^*W$. Under this isomorphism $\tilde{\nu}$ corresponds to a continuous section $\tilde{\sigma}\colon W\to T^*W$ that only intersects $T_X^*W$ at $(0,0)$. By \cite[Theorem 3.4]{LZ25-2} we have a microlocal intersection formula.
\begin{equation}
\label{eq; microlocal for Euobs}  
\textup{Eu}_{X}^{\tilde{\nu}}(0)=(-1)^{d}  \sharp_0 \left( [T_X^*W]\cdot [\tilde{\sigma}(W)]  \right) 
\end{equation} 
where $\sharp_0$ means the intersection number at $(0,0) \in T^*W$.
We then have 
\begin{prop}
\label{prop; nondegenerate}
Assume that $X$ has an isolated singularity at $0$.
If $\tilde{\nu}$ is tangent to  $X_{sm}$ and   has a non-degenerate isolated singularity at $0$, we have 
$
\Eu^{\tilde{\nu}}_X(0)=\Eu_X(0) \/.
$
\end{prop}
\begin{proof}
Let $E:=\sum_{i=0}^n   x_i\partial_{x_i}$ be the Euler vector field, then we have $\Eu^{E}_X(0)=\Eu_X(0)$ since $E$ is a radial vector field. 
Let $\cI_{\tilde{\nu}}:=(\tilde{\nu}_0,\tilde{\nu}_1, \cdots ,\tilde{\nu}_n )$ and  $\cI_E:= (x_0, x_1, \cdots ,x_n)=\mathfrak{m}_{W,0}$ be the Jacobian ideals of $\tilde{\nu}$ and $E$ in $\cO_{W,0}$ respectively, then we have
$\cI_{\tilde{\nu}}= \cI_E$ by the non-degenerate assumption.

Since the subscheme $Z_0$ of $Z$ is  cut by $\pi^*\mathfrak{m}_{X,0} = \pi^* (\cI_E\cO_{X,0})=\pi^* (\cI_{\tilde{\nu}}\cO_{X,0})$,
by the Gonz\'{a}lez-Sprinberg  formula we   have 
$
 \Eu_X(0)=\Eu^{E}_X(0)=\int_{Z_0} c(\cT)\cap s(Z_0, Z)=\Eu^{\tilde{\nu}}_X(0) \/.
$
\end{proof}

\subsection{Indices of Vector Fields}
Let $\tilde{\nu}$ be a holomorphic vector field on $W$ that is tangent to $D$. The $\GSV$ index $\GSV(\tilde{\nu}, D,0)$ was introduced  by X. Gomez-Mont, J. Seade and A. Verjovsky  in \cite{GSV} as a generalization of the Poincar\'e-Hopf index. The original definition concerns the obstruction to extending  $\tilde{\nu}|_D$ and the gradient $\overline{\nabla}_f$ to a local $2$-frame and we refer to \cite{GSV} for the precise definition. Later, the construction was extended to stratified vector fields on any complex hypersurface. The   readers  are referred to \cite[\S 3.5.2]{BSS87} for more details.

For simplicity we omit the definition of the  
$\GSV$ index in this paper. We will mostly rely on the following microlocal intersection formula proved in \cite[Corollary 5.4]{LZ25-2}. 
\begin{prop} 
\label{prop; microGSV}
Let $D$ and $\tilde{\nu}$ be as in Convention~\ref{conv}, then  we have
\begin{equation}
\label{eq; microGSV}
\GSV(\tilde{\nu}, D, 0)  = \Eu^{\tilde{\nu}}_D(0)   -\Eu_D(0)+ 1  +(-1)^{n}\mu.
\end{equation}
In particular, if $\tilde{\nu}$ has a non-degenerate isolated singularity at $0$, by Proposition~\ref{prop; nondegenerate} we have 
\[
\GSV(\tilde{\nu}, D, 0)  =  1  +(-1)^{n}\mu.
\]
\end{prop}

We also have the following algebraic formula from \cite[Theorem 1]{MR1642757}.
\begin{prop}
\label{prop; algebraicGSV}
Let $D$ and $\tilde{\nu}$ be as in Convention~\ref{conv}, we have
\begin{equation}
\GSV(\tilde{\nu}, D,0)=
\begin{cases}
\dim_\CC \frac{\cO_{W,0}}{(f, \tilde{\nu}_0,\tilde{\nu}_1, \cdots ,\tilde{\nu}_n  )} -\tau & \text{ if } n=2d+1 \\
\dim_\CC \frac{\cO_{W,0}}{( \tilde{\nu}_0,\tilde{\nu}_1, \cdots ,\tilde{\nu}_n  )} 
-\dim_\CC \frac{\cO_{W,0}}{\left(\frac{\tilde{\nu}(f)}{f}, \tilde{\nu}_0,\tilde{\nu}_1, \cdots ,\tilde{\nu}_n  \right)}
+\tau & \text{ if } n=2d   
\end{cases} \/.
\end{equation}
\end{prop}

Another local index  constantly studied, especially for logarithmic holomorphic vector fields along $D$, is 
 the logarithmic index $\Ind_{\log}(\tilde{\nu}, D, 0)$ introduced by Aleksandrov in \cite{AA05}. This index is defined as the Euler characteristic of the complex of logarithmic differential forms induced by contraction with $\tilde{\nu}$. Again,  for simplicity we will omit the precise definition and only recall the following two  formulas from \cite[Proposition 4.10 and Equation (4.1)]{LZ25-2}.
\begin{prop}
\label{prop; micrologFormulaiso}
Let $D$ and $\tilde{\nu}$ be as in Convention~\ref{conv}, then we have 
\begin{align}
\label{eq; logandtau}
\Ind_{\log}(\tilde{\nu}, D, 0)
&= \Ind_{\PH}(\tilde{\nu}, 0)
   - \GSV(\tilde{\nu}, D, 0)
   + (-1)^{n} \tau  \/, \\
\label{eq; microlog}
\Ind_{\log}(\tilde{\nu}, D, 0)
&= \sharp_0 \left(
[\textup{CC}(\ind_{W\setminus D})]\cdot [\tilde{\sigma}(W)]
\right)
+ (-1)^{n}\left(\tau - \mu \right) \/.
\end{align}
 \end{prop}

The following   corollary is then immediate. 
\begin{prop}
\label{prop; micrologcriteria}
The  following statements are equivalent.
\begin{enumerate}
	\item $D$ is weighted homogeneous at $0$.
	\item There exists a  holomorphic vector field $\tilde{\nu}$ on $W$ with an isolated singularity at $0$ and logarithmic along $D$, such that 
\[
\Ind_{\log}(\tilde{\nu}, D, 0)=\sharp_0 \left([\textup{CC}(\ind_{W\setminus D})]\cdot [\tilde{\sigma}(W)]  \right) \/.
\]
\item For every holomorphic vector field $\tilde{\nu}$ on $W$ with an isolated singularity at $0$ and logarithmic along $D$, we have
\[
\Ind_{\log}(\tilde{\nu}, D, 0)=\sharp_0 \left([\textup{CC}(\ind_{W\setminus D})]\cdot [\tilde{\sigma}(W)]  \right) \/.
\]
\end{enumerate}
\end{prop}

\section{Criteria of Weighted Homogeneity}
In this section we prove Conjecture~\ref{c1}. Recall that 
 $(W, 0)\subset \CC^{n+1}$ is the germ of an open domain and $D\subset W$ is the germ of a reduced hypersurface cut by $f\in \cO_{W,0}$. 
As in Convention~\ref{conv} 
 we assume that $D$ has   an isolated singularity at $0$ 
with Milnor number   $\mu$ and  Tjurina number  $\tau$. 
 Then $D$ is  weighted homogeneous at $0$ if and only if $\mu = \tau$.

\begin{theo}
\label{theo; non-degenerate}
The hypersurface germ $D$ has a weighted homogeneous singularity at $0$ if and only if there exists a logarithmic holomorphic vector field $\tilde{\nu}\in \Der_{W,0}(-\log D)$ such that $\tilde{\nu}$ has a   non-degenerate  isolated singularity at $0$.
\end{theo}
\begin{proof}
When $(D,0)$ is a weighted homogeneous singularity we may take the weighted Euler vector field 
 $\tilde{\nu}=\sum_{i=0}^n a_ix_i\partial_{x_i}$. Clearly this vector field is non-degenerate. 
 
Now we prove the other direction, where we assume  
that $\tilde{\nu} =\sum_{i=0}^{n} \tilde{\nu}_i \partial_{x_i}\in \Der_{W,0}(-\log D)$ is a logarithmic holomorphic vector field with a non-degenerate isolated  singularity at $0$. Then we have  
\[
\Ind_{\PH}(\tilde{\nu}, 0)=\dim_\CC \frac{\cO_{W,0}}{(\tilde{\nu}_0, \cdots, \tilde{\nu}_{n})}=\dim_\CC \frac{\cO_{W,0}}{\mathfrak{m}_{W,0}}=1 \/.
\]

First we assume that $n=2d+1$ is odd.  By \cite[Theorem 1]{MR1642757}, as $f(0)=0$ we have
\[
\GSV(\tilde{\nu}, D, 0)=\dim_\CC \frac{\cO_{W,0}}{(f, \tilde{\nu}_0, \cdots, \tilde{\nu}_{2d+1})} -  \tau =1-\tau \/.
\]

Combining Equation \eqref{eq; logandtau} and \eqref{eq; microlog} we have 
\begin{align*}
0=\GSV(\tilde{\nu}, D, 0)-(1-\tau)
=& \Ind_{\PH}(\tilde{\nu}, 0)-1-\Ind_{\log}(\tilde{\nu}, D, 0) = -\Ind_{\log}(\tilde{\nu}, D, 0) \\
=&  -\sharp_0 \left([\textup{CC}(\ind_{W\setminus D})]\cdot [\tilde{\sigma}(W)]  \right)+ \left(\tau -\mu \right) \/.
\end{align*}
Since $D$ has only isolated singularity at $0$,   
\begin{align*}
\ind_{W\setminus D}
=&\ind_W-\ind_D  = \ind_W-\Eu_D+(\Eu_D(0) -1)\ind_{\{0\}} \\
=&(-1)^{n+1}\Eu^\vee_W
-(-1)^n\Eu^\vee_D +(\Eu_D(0)-1)\Eu^\vee_{\{0\}} \\
=& \Eu^\vee_W
+\Eu^\vee_D +(\Eu_D(0)-1)\Eu^\vee_{\{0\}} \/.
\end{align*}
As $\tilde{\nu}$ is non-degenerate, we see that $(-1)^{n+1}\sharp_0 \left([T_W^*W]\cdot [\tilde{\sigma}(W)]\right) =\Ind_{\PH}(\tilde{\nu},0)=1$. Then
\begin{align}
\sharp_0 \left([\textup{CC}(\ind_{W\setminus D})]\cdot [\tilde{\sigma}(W)]  \right)
=& (-1)^{n+1} \sharp_0  \left([T_W^*W]\cdot [\tilde{\sigma}(W)]  \right)
+(-1)^{n+1} \sharp_0  \left([T_D^*W]\cdot [\tilde{\sigma}(W)]  \right) \nonumber \\
& +(\Eu_D(0)-1)\cdot  \sharp_0  \left([T_{\{0\}}^*W]\cdot [\tilde{\sigma}(W)]  \right) \nonumber \\
=& 1+(-1)^{n+1} \sharp_0  \left([T_D^*W]\cdot [\tilde{\sigma}(W)]  \right) +(\Eu_D(0)-1) \nonumber \\
=&  \sharp_0  \left([T_D^*W]\cdot [\tilde{\sigma}(W)]  \right) + \Eu_D(0) 
\label{eq; microintersection} \/.
\end{align} 
Therefore we have
\begin{align*}
0=\GSV(\tilde{\nu}, D, 0)-(1-\tau)
=&   
-\sharp_0 \left([\textup{CC}(\ind_{W\setminus D})]\cdot [\tilde{\sigma}(W)]  \right)+\left(\tau-\mu \right) \\
=&  (-1)^n\sharp_0  \left([T_D^*W]\cdot [\tilde{\sigma}(W)]  \right)-\Eu_D(0) +\left(\tau-\mu \right)  \\
=& \Eu^{\tilde{\nu}}_D(0) -\Eu_D(0) +\left(\tau-\mu \right)   \/.
\end{align*}
The last equality is due to Equation~\eqref{eq; microlocal for Euobs}. 
Since $\tilde{\nu}$ is non-degenerate,  we have $\Eu^{\tilde{\nu}}_D(0) =\Eu_D(0)$ by Proposition~\ref{prop; nondegenerate}. Therefore
$\mu=\tau$ and   $D$ is weighted homogeneous.

The case of  $n=2d$ has been proved in \cite[Theorem 2]{Machado-Seade26} based on the lower bound of \cite[Theorem 1]{Machado-Seade26}.  Here we provide an alternating proof without adapting the lower bound. Since $n=2d$ is even, by \cite[Theorem 1]{MR1642757}, setting $\tilde{\nu}(f)=hf$ we have 
\[
\GSV(\tilde{\nu}, D, 0)=\tau
+
\dim_\CC \frac{\cO_{W,0}}{( \tilde{\nu}_0, \cdots, \tilde{\nu}_{2d})} -\dim_{\CC}\frac{\cO_{W,0}}{(h, \tilde{\nu}_0, \cdots, \tilde{\nu}_{2d})} =1+ \tau- \dim_{\CC}\frac{\cO_{W,0}}{(h, \tilde{\nu}_0, \cdots, \tilde{\nu}_{2d})}\/.
\]
Since   $(\tilde{\nu}_0, \cdots, \tilde{\nu}_{2d})=\mathfrak{m}_{W,0}$, we have
$\dim_{\CC}\frac{\cO_{W,0}}{(h, \tilde{\nu}_0, \cdots, \tilde{\nu}_{2d})}=1- \dim_\CC  h(0) \CC$ is either $0$ or $1$, depending on whether $h$ is a unit in $\cO_{W,0}$. 
By   Proposition~\ref{prop; microGSV}   we have
\[
\GSV(\tilde{\nu}, D, 0)=1+\mu+ \Eu^{\tilde{\nu}}_D(0)   - \Eu_D(0)=1+\mu   \/.
\]
Combining previous two equations we have 
\[
1+\mu=1+\tau- \dim_{\CC}\frac{\cO_{W,0}}{(h, \tilde{\nu}_0, \cdots, \tilde{\nu}_{2d})}\leq 1+\tau \/.
\]
Since $\mu\geq \tau$ we then obtain 
$\tau=\mu$, and $D$ is weighted homogeneous at $0$.
\end{proof}

Tracking the computation in \eqref{eq; microintersection} we see that
\begin{coro}
\label{coro; logindnondegerate}
Let $D$ and $\tilde{\nu}$ be as in Convention~\ref{conv}. 
If $\tilde{\nu}$ has a    non-degenerate 	isolated singularity at $0$, then  we have
\[
\Ind_{\log}(\tilde{\nu}, D, 0)=0 \/.
\]
\end{coro}
 When $D$ is a free divisor this is known by \cite[Proposition 2]{AA05}. Thus this corollary gives a parallel analogue for isolated hypersurface singularities.

The following corollary is a consequence of Theorem~\ref{theo; non-degenerate} and \cite[Theorem A]{Morvan24}.
\begin{coro}
\label{coro; conjsolved} 
Conjecture~\ref{c1} is true.
\end{coro}
\begin{proof}
It remains to prove that, if   $\tilde{\nu}\in \Der_{W,0}(-\log D)$ has an isolated singularity at $0$ and 
is transverse to  the link spheres of $D$, then $D$ is weighted homogeneous. By Definition~\ref{defi; transverse}, there exists $\epsilon_0>0$ such that   
$dr(\text{Re}(\tilde{\nu}))(P)    \neq 0$ for any $P\in B_{\epsilon_0}(0)\setminus \{0\}$. Replacing $\tilde{\nu}$ by $-\tilde{\nu}$ if necessary  we may assume that $dr(\text{Re}(\tilde{\nu}))(P)  < 0$ on $B_{\epsilon_0}(0)\setminus \{0\}$, i.e., the vector field $\tilde{\nu}$   points inwards everywhere on the link spheres. In particular $\Sing(\text{Re}(\tilde{\nu}))=\{0\}$.

Let $\Phi_t$ be the local holomorphic flow of $\tilde{\nu}$  defined in $B_{\epsilon_1}(0)$ for some $\epsilon_0>\epsilon_1 >0$.  
Since $\tilde{\nu}$   is tangent to $D$ and points inwards, the flow $\Phi_t$ takes 
$D\cap  B_\epsilon(0)  \setminus \{0\}$ to $D\cap B_\epsilon(0)\setminus \{0\}$ for every $\epsilon <\epsilon_1$. 
Since $\Sing(\text{Re}(\tilde{\nu}))=\{0\}$, the flow $\Phi_t$ then produces a contracting automorphism $F$ on the germ $(D,0)$ in the sense of \cite[Definition 2.1]{Morvan24}. 
By \cite[Theorem A]{Morvan24} the hypersurface $D$ has be to weighted homogeneous at $0$.
\end{proof}

\begin{rema}
\label{rema}
In fact in the proof of  \cite[Theorem 3.3]{Morvan24} the author states that all the eigenvalues of the Jacobian matrix $DF(0)$ have length $<1$. Therefore the eigenvalues of the linear part of $\tilde{\nu}$ have negative real parts and  $\tilde{\nu}$  is non-degenerate.  
 Thus on an isolated hypersurface germ, 
if a holomorphic vector field  tangent to $D$ (with an isolated singularity at $0$) is transverse to  the link spheres of $D$, it must have a non-degenerate isolated singularity at $0$. For more details we refer to the self-contained proof generated by the Danus system in \cite{LiuZ26}. 
\end{rema}

\section{A Residue Interpolation of the GSV Index}
\label{sec; Residue}
Now we discuss when the lower bound in Theorem~\ref{theo; MSbound} is achieved. From Example~\ref{counterexam} we see that it holds for  degenerate vector fields. 
We retain the notation and hypotheses of $D$, $W$ and $\tilde{\nu}$ as in Convention~\ref{conv}. We write $\tilde{\nu}=\sum_{i=0}^n \tilde{\nu}_i \partial_{x_i}$ and  $h=\frac{\tilde{\nu}(f)}{f}=\frac{\sum_{i=0}^n \tilde{\nu}_i \cdot f_{x_i}}{f}$. 

First we review Suwa's residue formula for the $\GSV$ index from \cite{Suw96}.
Set $\cI_{\tilde{\nu}}=(\tilde{\nu}_0, \tilde{\nu}_1, \cdots, \tilde{\nu}_n)$ and
let $\fm=(x_0, x_1, \cdots, x_n)$ be the maximal ideal of $\cO_{W, 0}$.
Then isolated singularity assumption implies that
$\{\tilde{\nu}_0, \tilde{\nu}_1, \cdots, \tilde{\nu}_n\}$ is regular and
$\cI_{\tilde{\nu}}$ is $\fm$-primary. Moreover, by \cite[Lemma 1]{Suw96} we may choose the local coordinates such that 
that $\{\tilde{\nu}_1, \tilde{\nu}_2, \cdots, \tilde{\nu}_{n}, f\}$ is also regular.

%We  denote by $\underline{\tilde{\nu}}$ the sequence
%$\{\tilde{\nu}_0, \tilde{\nu}_1, \tilde{\nu}_2, \cdots, \tilde{\nu}_n\}$
%and by $\underline{\tilde{\nu}_f}$ the sequence
%$\{\tilde{\nu}_1, \tilde{\nu}_2, \cdots, \tilde{\nu}_{n}, f\}$.

Let
$M_{\tilde{\nu}}:=\begin{bmatrix}
\frac{\partial \tilde{\nu}_i}{\partial x_j}
\end{bmatrix}_{0\leq i,j\leq n}$
be the Jacobian matrix of $\tilde{\nu}$.
Following the curvature definition for the Chern classes
(see \cite{MR440554}), we define holomorphic functions
$c_j(M_{\tilde{\nu}})$ as the coefficients in the expansion 
\[
\det\left(I_{n+1}+\frac{t}{2\pi i}M_{\tilde{\nu}}\right)
=c_0(M_{\tilde{\nu}})
+c_1(M_{\tilde{\nu}})t
+c_2(M_{\tilde{\nu}})t^2
+\cdots
+c_{n+1}(M_{\tilde{\nu}})t^{n+1} \/.
\]
\begin{exam}
\label{exam; c0cn}
Let $M^*_{\tilde{\nu}}$ be the  classical adjugate  matrix of $M_{\tilde{\nu}}$. Then we have
\[
c_0(M_{\tilde{\nu}})=1, \quad
c_{n+1}(M_{\tilde{\nu}})=\frac{1}{(2\pi i)^{n+1}} \det M_{\tilde{\nu}}\/,
\quad 
c_n(M_{\tilde{\nu}})
=
\frac{1}{(2\pi i)^n}  \tr 
M^*_{\tilde{\nu}}  \/.
\] 
\end{exam}

Take the unnormalized coefficients
$\sigma_k(M_{\tilde{\nu}}):=(2\pi i)^k c_k(M_{\tilde{\nu}})$,
we define   holomorphic functions
\begin{equation}
\label{eq; Pk}
\cP_k(\tilde{\nu},h)
:=\sigma_0(M_{\tilde{\nu}})\cdot(-h)^k
+\sigma_1(M_{\tilde{\nu}})\cdot(-h)^{k-1}
+\cdots
\sigma_{k-1}(M_{\tilde{\nu}})\cdot(-h)
+\sigma_k(M_{\tilde{\nu}}) \/.
\end{equation}
This is the coefficient of of $t^k$ in the formal power series expansion of
$\frac{\det\left(I_{n+1}+tM_{\tilde{\nu}}\right)}{1+th}$.

\begin{theo}[{\cite[Theorem 1]{Suw96}}]
\label{theo; Suwa}
Let $\Gamma_\epsilon$ be the $n$-cycle in $W$ given by
\[
\Gamma_\epsilon =\left\lbrace
x\ \vert\
|\tilde{\nu}_1(x)|
=\epsilon_1 \/,  |\tilde{\nu}_2(x)|
=\epsilon_2 \/, \cdots \/,
|\tilde{\nu}_n(x)|
=\epsilon_n
\text{ and } f(x)=0
\right\rbrace
\]
for some small enough positive real numbers $\{\epsilon_i\}$.
We orient $\Gamma_\epsilon$ so that
$d\theta_1\wedge\cdots\wedge d\theta_n>0$ is positive,
where $\theta_k:=\arg\tilde{\nu}_k$. Let $\omega_0:=dx_1 \wedge \cdots \wedge dx_n$, 
then we have
\[
\GSV(\tilde{\nu},D,0)
=\frac{1}{(2\pi i)^{n}}
\int_{\Gamma_\epsilon}
\frac{\cP_n(\tilde{\nu}, h)\cdot \omega_0}
{\tilde{\nu}_1 \tilde{\nu}_2 \cdots\tilde{\nu}_n} \/.
\]
\end{theo}

The integration formula on the RHS is the residue at $0$ of the meromorphic form
 $\frac{\cP_n(\tilde{\nu}, h)\cdot\omega_0}{\tilde{\nu}_1\cdots \tilde{\nu}_n}$, where we denote it by
$\Res_0
\left[
\begin{array}{c}
\cP_n(\tilde{\nu}, h)\cdot\omega_0 \\
\tilde{\nu}_1,\ldots,\tilde{\nu}_n
\end{array}
\right]$  using the Grothendieck residue symbol.

The residues play important
role in the following framework.
Set $A:=\frac{\cO_{W,0}}{\cI_{\tilde{\nu}}}$. 
Since $\Sing(\tilde{\nu})=\{0\}$,  $\Spec(A)$ is a zero-dimensional complete intersection and therefore $A$ is an Artinian Gorenstein local $\mathbb C$-algebra. We consider  the    class $[h]$ of $h=\frac{\tilde{\nu}(f)}{f}$ in $A$.

By local duality
\cite[(1.5)]{GH78res}, the Grothendieck residue
induces a nondegenerate pairing on $A$:
\[
\langle \alpha\, \vert \,\beta \rangle:= 
\Res_0
\left[
\begin{array}{c}
\alpha\beta\cdot \omega\\
\tilde{\nu}_0,\cdots,\tilde{\nu}_n
\end{array}
\right]
=
\frac{1}{(2\pi i)^{n+1}}
\int_{T_\epsilon} 
\frac{\alpha\beta\cdot \omega}{\tilde{\nu}_0\cdots \tilde{\nu}_n} \/,
\]
where $\omega:=dx_0\wedge \omega_0=dx_0\wedge dx_1\wedge\cdots \wedge dx_n$ and $T_\epsilon$ denotes the $n+1$-cycle give by 
\[
T_\epsilon:=\{
x\in W \vert \ \|\tilde{\nu}_0(x)\|=\epsilon_0\/, \|\tilde{\nu}_1(x)\|=\epsilon_1\/, \cdots , \|\tilde{\nu}_n(x)\|=\epsilon_n
\}  \/,
\]
oriented so that
$d\theta_0\wedge\cdots\wedge d\theta_n>0$.

The Grothendieck residue symbols satisfy the following transformation
law by \cite[(1.4)]{GH78res}.  Given any two 
regular sequences $\underline{u}=(u_0,   \cdots ,u_n)$ and $\underline{g}=(g_0,  \cdots ,g_n)$ in $\cO_{W,0}$ such that $\underline{g}=\Phi \underline{u}$ for some $\Phi\in \operatorname{Mat}_{n+1}(\cO_{W,0})$, we have 
$\Res_0
\left[
\begin{array}{c}
 \alpha\cdot  \omega\\
u_0, u_1,\cdots, u_n
\end{array}
\right]
=
  \Res_0
\left[
\begin{array}{c}
 \det(\Phi) \alpha\cdot  \omega\\
g_0, g_1,\cdots, g_n
\end{array}
\right]$.
Moreover,  we have $ 
\Res_0
\left[
\begin{array}{c}
 \alpha\cdot  \omega\\
u_0, u_1,\cdots, u_n
\end{array}
\right]
=0$  whenever
$\alpha \in (u_0,\cdots, u_n)$.

Let $\mathfrak{m}_A=\mathfrak{m}A$ be the maximal ideal of $A$.
The socle of $A$ is  defined as the orthogonal complement of $\mathfrak{m}_A$ with respect to the Grothendieck residue pairing:
\[
\textup{Soc}(A):= \{\alpha\in A \,\vert \langle \alpha \,\vert\, g \rangle=0 \text{ for all } g\in \mathfrak{m}_A \} \/.
\] 
Since $A$ is Artinian Gorenstein local, the non-degeneracy of the Grothendiecl pairing  implies that $\textup{Soc}(A)=\textup{Ann}_A(\fm_A)$ is an ideal of $A$.   Moreover this is the unique nonzero ideal of $A$ that is $1$-dimensional as a $\CC$-vector space, since each non-zero ideal of $A$ contains some non-zero element of $\textup{Soc}(A)$ (see \cite[Proposition 21.5]{EisenbudCA}).

We now translate Suwa's residue formula on $\Gamma_\epsilon$ into the following residue formula on $T_\epsilon$.

\begin{prop}
\label{prop; residueGSV}
Under a  suitable local analytic coordinates, we have
\begin{equation}
\label{eq; residue}
\GSV(\tilde{\nu}, D, 0)= \Res_0
\left[
\begin{array}{c}
h\cP_n(\tilde{\nu},h)\,\omega\\
\tilde{\nu}_0,\cdots,\tilde{\nu}_n
\end{array}
\right] 
= \frac{1}{(2\pi i)^{n+1}}
\int_{T_\epsilon} 
\frac{h  \cP_n(\tilde{\nu}, h)\cdot dx_0\wedge  \cdots \wedge dx_n}{\tilde{\nu}_0\tilde{\nu}_1\cdots  \tilde{\nu}_n}   \/.
\end{equation}
\end{prop}
 
\begin{proof}
Choose coordinates such that the ideal $\cI_{\tilde{\nu}_f}=(f,\tilde{\nu}_1,\cdots,\tilde{\nu}_n)$ is $\fm$-primary 
as in \cite[Lemma~1]{Suw96}. Then by Theorem~\ref{theo; Suwa} we have
\[
\GSV(\tilde{\nu},D,0)
=
\Res_0
\left[
\begin{array}{c}
\cP_n(\tilde{\nu},h)\cdot\omega_0\\
\tilde{\nu}_1,\cdots,\tilde{\nu}_n
\end{array}
\right]
=\frac{1}{(2\pi i)^{n}}
\int_{\Gamma_\epsilon}
\frac{\cP_n(\tilde{\nu}, h)\cdot \omega_0}
{\tilde{\nu}_1 \tilde{\nu}_2 \cdots\tilde{\nu}_n} \/.
\]

We recall  the Leray residue formula from
\cite[Chap 2, \S 21]{Ler59}. 
For the $n$-cycle $\Gamma_\epsilon\subset D_{sm}$, we define its Leray tube  as
the real $(n+1)$-cycle, for $\delta$ small enough
\[
T_\delta(\Gamma_{\epsilon})
:=
\left\{
x\in W \, \vert \,
|f(x)|=\delta \text{ and }
|\tilde{\nu}_j(x)|=\epsilon_j \text{ for any } 1\leq j\leq n
\right\},
\]
with   orientation chosen such that
$d\arg f\wedge d\arg\tilde{\nu}_1\wedge\cdots
\wedge d\arg\tilde{\nu}_n >0$.

Notice that $df\wedge \omega_0= f_{x_0}\omega$, then 
by the Leray residue formula we have
\begin{align*}
\frac{ 1}{(2\pi i)^{n}}
\int_{\Gamma_\epsilon}
\frac{\cP_n(\tilde{\nu}, h)\cdot \omega_0}
{\tilde{\nu}_1 \tilde{\nu}_2 \cdots\tilde{\nu}_n} 
=&
\frac{1}{(2\pi i)^{n+1}}
\int_{T_{\delta}(\Gamma_\epsilon)}
\frac{df}{f}\wedge
\frac{\cP_n(\tilde{\nu}, h)\cdot \omega_0}
{\tilde{\nu}_1 \tilde{\nu}_2 \cdots\tilde{\nu}_n}  \\
=&\frac{1}{(2\pi i)^{n+1}}
\int_{T_{\delta}(\Gamma_\epsilon)} 
\frac{f_{x_0}  \cP_n(\tilde{\nu}, h)\cdot \omega}
{f   \tilde{\nu}_1   \cdots \tilde{\nu}_n} \\
=& \Res_0
\left[
\begin{array}{c}
f_{x_0}\cP_n(\tilde{\nu},h)\cdot \omega\\
f,\tilde{\nu}_1,\cdots,\tilde{\nu}_n
\end{array}
\right]\/.
\end{align*}
The last equality is because    
  both sequences 
$\underline{\tilde{\nu}}_f=(f,\tilde{\nu}_1,\cdots,\tilde{\nu}_n)$ and
$\underline{\tilde{\nu}}=(\tilde{\nu}_0,\cdots,\tilde{\nu}_n)$ are regular by  assumptions. Moreover  the sequence
$\underline{\tilde{\nu}'}_f:=(f\tilde{\nu}_0,\tilde{\nu}_1,\cdots,\tilde{\nu}_n)$   is also   regular.  
The determinant transforming  from $\underline{\tilde{\nu}}$ and $\underline{\tilde{\nu}}_f$ to $\underline{\tilde{\nu}'}_f$ are $f$ and $\tilde{\nu}_0$ respectively. 
Then we have 
\begin{align*}
\Res_0
\left[
\begin{array}{c}
f_{x_0}\cP_n(\tilde{\nu},h)\cdot \omega\\
f,\tilde{\nu}_1,\cdots,\tilde{\nu}_n
\end{array}
\right]
=& 	
\Res_0
\left[
\begin{array}{c}
f_{x_0}\tilde{\nu}_0\cP_n(\tilde{\nu},h)\cdot \omega\\
f\tilde{\nu}_0,\tilde{\nu}_1,\cdots,\tilde{\nu}_n
\end{array}
\right]   
= 
\Res_0
\left[
\begin{array}{c}
\left(fh-\sum_{i=1}^n\tilde{\nu}_i f_{x_i} \right)\cP_n(\tilde{\nu},h)\cdot \omega\\
f\tilde{\nu}_0,\tilde{\nu}_1,\cdots,\tilde{\nu}_n
\end{array}
\right]\\
=&  
\Res_0
\left[
\begin{array}{c}
fh\cP_n(\tilde{\nu},h)\cdot \omega\\
f\tilde{\nu}_0,\tilde{\nu}_1,\cdots,\tilde{\nu}_n
\end{array}
\right] 
= \Res_0
\left[
\begin{array}{c}
h\cP_n(\tilde{\nu},h)\cdot \omega\\
\tilde{\nu}_0,\tilde{\nu}_1,\cdots,\tilde{\nu}_n
\end{array}
\right]\\
=& \frac{1}{(2\pi i)^{n+1}}
\int_{T_\epsilon} 
\frac{h\cdot \cP_n(\tilde{\nu}, h)\cdot \omega}{\tilde{\nu}_0\tilde{\nu}_1\cdots  \tilde{\nu}_n} 
 \/.    
\end{align*} 

\end{proof}

\begin{prop}
\label{prop; ranklinearpart}
Assume that $h(0)=0$ and $0\neq[h]\in\Soc(A)$. If $n$ is even,  then 
$\rk \left(M_{\tilde{\nu}}(0)\right)=n$.
\end{prop}

\begin{proof}
Since $[h]\neq 0$, 
$([h])=\CC\cdot [h]=\Soc(A)=\textup{Ann}_A(\fm_A)$ and  we have, for every $g\in \cO_{W,0}$, 
\[
[hg]-g(0)[h]=[h(g-g(0))]\in [h]\cdot \fm_A=0 \/.
\]
Thus $[hg]=g(0)[h]$ for every $g\in \cO_{W,0}$.
In particular $[h]^2=h(0)[h]=0$.
Recall \eqref{eq; Pk},  the definition of $\cP_n(\tilde{\nu}, h)$,  we have
\[
[h\cP_n(\tilde{\nu},h)]=\left[h\left(\sigma_{n}(M_{\tilde{\nu}})-h\sigma_{n-1}(M_{\tilde{\nu}})+ O(h^2)   \right)\right]
= \sigma_n (M_{\tilde{\nu}})(0)\cdot [h]   \/.\]
Substituting into Proposition~\ref{prop; residueGSV}  we have 
\[
\GSV(\tilde{\nu}, D, 0)=
\Res_0
\left[
\begin{array}{c}
h\cP_n(\tilde{\nu},h)\cdot \omega\\
\tilde{\nu}_0,\tilde{\nu}_1,\cdots,\tilde{\nu}_n
\end{array}
\right]=\langle [h\cP_n(\tilde{\nu},h)]\, \vert \, 1\rangle 
= \sigma_n (M_{\tilde{\nu}})(0)\cdot \langle  [h]\, \vert \, 1\rangle \/.
\]
Since $n$ is even, combining previous euqation with Proposition~\ref{prop; algebraicGSV} we have 
\begin{equation}
\label{eq; socleresidue}
 \GSV(\tilde{\nu}, D, 0)
 =\sigma_n (M_{\tilde{\nu}})(0) \cdot 
\Res_0
\left[
\begin{array}{c}
h \cdot \omega\\
\tilde{\nu}_0, \cdots,\tilde{\nu}_n
\end{array} 
\right]
=\dim_\CC ([h]) +\tau= \tau +1 \geq 1 \/.
\end{equation}
Consequently $\sigma_n (M_{\tilde{\nu}})(0)= 
\sigma_n (M_{\tilde{\nu}}(0))  \neq 0$.

Recall from Example~\ref{exam; c0cn} that $\sigma_n (M_{\tilde{\nu}})=\tr(M^*_{\tilde{\nu}})$, whose evaluation at $0$ appears in the following identity 
\[
\det(M_{\tilde{\nu}}(0)+tI_{n+1})=\det M_{\tilde{\nu}}(0)+ t\tr(M^*_{\tilde{\nu}}(0))+O(t^2),
\]

If $\det M_{\tilde{\nu}}(0) \neq 0$, by implicit function theorem we have
$\cI_{\tilde{\nu}}= \fm$ and therefore $[h]= 0$. But this contradicts our assumption, thus 
$\det M_{\tilde{\nu}}(0) =0$ and consequently $0$ is an eigenvalue of multiplicity precisely $1$.  
\end{proof}

We have a direct application when $n$ is even and $\tilde{\nu}$ is non-degenerate. 
\begin{prop}
Let $D$, $W$ and $\tilde{\nu}$ be as in Convention~\ref{conv} and we assume $n$ is even. 
If $\tilde{\nu}$ has a non-degenerate isolated singularity at $0$, then 
 $[h]$ is a socle generator in $A$ and we have
\[
\frac{1}{(2\pi i)^{n+1}}\int_{T_\epsilon} \frac{h\cdot \omega}{\tilde{\nu}_0\tilde{\nu}_1\cdots \tilde{\nu}_n}=
 \frac{(1+\tau)}{\sigma_{n}(M_{\tilde{\nu}})(0)} =
 \frac{(1+\mu)}{\sigma_{n}(M_{\tilde{\nu}})(0)} 
  \neq 0 \ 
\text{ and }\ 
\int_{T_\epsilon}  
 \frac{hg\cdot \omega}{\tilde{\nu}_0\tilde{\nu}_1\cdots \tilde{\nu}_n}
=0 \text{ for any } g\in \fm \/.
\]  
\end{prop} 
\begin{proof}
From the proof of Theorem~\ref{theo; non-degenerate} we see that,  by Proposition~\ref{prop; algebraicGSV},   
\[
\dim_{\CC} \frac{\cO_{W,0}}{(h, \tilde{\nu}_0, \tilde{\nu}_1, \cdots, \tilde{\nu}_n)}
=\dim_{\CC} \frac{A}{([h])}
= \dim_{\CC} A-1 \/,
\]  
Thus we have   $\dim_{\CC} ([h])=1$ and therefore 
$[h]$ generates the socle ideal of $A$. The second equality then comes from the definition of socle ideal.

Since $\tilde{\nu}$ is non-degenerate, by Theorem~\ref{theo; non-degenerate} we have $(1+\mu)=(1+\tau)=\GSV(\tilde{\nu}, D, 0)$. The first equality then comes
 from the proof of Proposition~\ref{prop; ranklinearpart}, where
\[
(1+\tau)=\GSV(\tilde{\nu}, D, 0)
= \sigma_n (M_{\tilde{\nu}})(0)\cdot \langle  [h]\, \vert \, 1\rangle 
=  \frac{\sigma_n (M_{\tilde{\nu}})(0)}{(2\pi i)^{n+1} }  \int_{T_\epsilon} \frac{h\cdot \omega}{\tilde{\nu}_0\tilde{\nu}_1\cdots \tilde{\nu}_n}
\/. \qedhere
\]
\end{proof}
It is   a natural and interesting question to ask whether  $\sigma_{n}(M_{\tilde{\nu}})(0)$ measure any topological invariant of the vector field $\tilde{\nu}$ along $D$. 
 
\section{Attaining the Lower Bound of the $\GSV$ Index}
\label{sec; hessian}
Finally in this section we prove the equality criterion announced in Theorem~\ref{theo; GSVlowerboundintro}.

First we utilize \cite[Theorem~5.3, Theorem~5.4]{GS06}and their proofs  to derive
 a lower bound for the rank of the Hessian
matrix of $f$, in terms of the Jacobian matrix $M_{\tilde{\nu}}(0)$ of any logarithmic vector field $\tilde{\nu}$ along $D$. 

Let $M_{\tilde{\nu}}(0)=S+N$ be the semi-simple decomposition of the matrix $M_{\tilde{\nu}}(0)$, where $S$ denotes the semisimple part and $N$ denotes the nilpotent part. We have $[S, N]=0$.

\begin{lemm}
\label{lemm; Hessianrank}
Let $\Hess_f=\begin{bmatrix}
	\frac{\partial^2 f}{\partial x_i\partial x_j}
\end{bmatrix}_{0\leq i, j\leq n}$ be the Hessian matrix of $f$.  
 Let  $\tilde{\nu}$ be any logarithmic vector field germ satisfying \textup{Convention}~\ref{conv}. If  $h(0)=0$,  
we have 
\[
\rk M_{\tilde{\nu}}(0) \leq\rk\Hess_f(0).
\]
\end{lemm}

\begin{proof}
We adapt the setting of \cite{GS06} and consider the $\fm$-adic completion  
$\widehat{\cO_{W,0}}:=\CC[[x_0,\cdots,x_n]]$, though 
we slightly modify the convention from \cite{GS06} and \cite{KSaito71}.  

Up to a linear   coordinate change we  assume that 
$S=\operatorname{diag}(w_0, w_1, \cdots, w_n)$.
We set
\[
\nu_S:= \underline{x}^T S \overline{\partial_{\underline{x}}}
=\sum_{i=0}^n w_i x_i\partial_{x_i},\quad
\nu_N:= \underline{x}^T N \overline{\partial_{\underline{x}}}
=\sum_{i,j=0}^{n}N_{ij}x_j\partial_{x_i} \/.
\]
Then we have $[\nu_S, \nu_N]=0$.
Since $\nu_S$ is homogeneous of degree $1$, any  monomial $\underline{x}^\alpha=x_0^{\alpha_0}x_1^{\alpha_1}\cdots x_n^{\alpha_n}$ is necessarily an eigenvector of $\nu_S$. 
We call the   eigenvalue $\sum_{i=0}^n w_i\alpha_i$ the weight of $\underline{x}^\alpha$ and denote it by $\langle \underline{w},\alpha\rangle$.
By the degree of $\underline{x}^\alpha$ we mean $|\alpha|:=\sum_{i=0}^n \alpha_i$. 

We now take the trivial multigrading
$s=1$ and $\underline{w^1}=(0,\ldots,0)$ in
\cite[Theorem~5.3]{GS06}. Under this multigrading, every formal coordinate change is $W$-homogeneous and every
formal power series, every formal vector field has degree $0$ in the sense of  \cite[\S 5]{GS06}.
By \cite[Theorem~5.3]{GS06}, there exists a formal
coordinate change $\underline{x}=\phi(\underline{y})$ such that  the transformed vector field
$\tilde{\chi}:=\phi^*\tilde{\nu}$ satisfies
\[
\tilde{\chi}=\nu_S+\nu_N+\nu_R,
\quad
[\nu_S,\tilde{\chi}]=[\nu_S, \nu_R]=[\nu_S, \nu_N]=0, 
\quad 
\nu_R\in \widehat{\fm}_y^2\cdot \widehat{\Der}_{W,0}.
\]
Here $\widehat{\fm}_y=(y_0,\ldots,y_n)\subset\CC[[\underline{y}]]$.
The construction of  $\phi$ given in the proof of  \cite[Theorem~5.3]{GS06} is successive,
namely the authors construct a coordinate change that fulfills requirements up to order $m$ and modify the order $m+1$ term using the next coordinate change. 
Since we fix the linear part $S+N$,  $\phi$ then may be chosen so that $\phi(0)=0$ and $d\phi(0)=I_{n+1}$.

For any $d\geq 0$ and $\lambda\in\CC$, we define $\nu_S$-invariant subspaces
\[
V_{d,\lambda}
=
\operatorname{Span}_{\CC}
\left\{
\underline{y}^{\alpha} \, \vert\,
|\alpha|=d,\ 
\langle\underline{w},\alpha\rangle=\lambda
\right\}.
\]
These are finite dimensional $\CC$-vector spaces.
Since $[S,N]=0$ and $N^{n+1}=0$, the vector field  $\nu_N$   acts nilpotently on
 the spaces $V_{d,\lambda}$.

It's obvious that
$\tilde{\chi}(\widetilde{f})=\widetilde{h}\cdot\widetilde{f}$ for the transformed functions $\widetilde{f}:=f\circ\phi$ and
$\widetilde{h}=h\circ\phi$.
Let $\widetilde{h}^{(0)}$  be the (formal) sum of all monomial appearing in $\widetilde{h}$ that has weight $0$ and $\widetilde{h}^{(\ne 0)}:=\widetilde{h}-\widetilde{h}^{(0)}$ be the (formal)  sum of the rest monomials of $\widetilde{h}$. 
Since   constant term has weight $0$, 
we have $\widetilde{h}^{(\neq 0)}(0)=0$ and therefore 
 $\widetilde{h}^{(0)}(0)=\widetilde{h}(0)=h(0)=0$.

If $\underline{x}^\alpha\in V_{d, \lambda}$ has weight $\lambda\neq 0$,
$\nu_S+\nu_N=\nu_{S+N}$ is invertible and there exists some $g^\alpha_{\lambda}\in V_{d, \lambda}$ such that $(\nu_S+\nu_N)(g^\alpha_{\lambda})=\underline{x}^\alpha$.
Since $[\nu_S, \nu_R]=0$ and $\nu_R\in \widehat{\fm}_y^2\widehat{\Der}_{W,0}$,  the vector field 
  $\nu_R$ preserves the weight of a monomial but  strictly raises its degree. Consequently we have
 \[
 \tilde{\chi}(g_\lambda^\alpha)-\underline{x}^\alpha
 = \nu_R(g_\lambda^\alpha) \in V_{>d, \lambda}:=\prod_{d'=d+1}^\infty V_{d', \lambda} \/.
 \] 
Therefore by  induction on the degrees of the monomials appearing in $\widetilde{h}$, there exists some $g\in \widehat{\cO_{W,0}}$ such that $g(0)=0$ and 
$\tilde{\chi}(g)=\widetilde{h}^{(\neq 0)}$.

Finally we consider the unit $u=e^{-g}$ and define $F=u\widetilde{f}$. Then
\[
\tilde{\chi}(F)=\tilde{\chi}(u)\widetilde{f}+ u\tilde{\chi}(\widetilde{f}) 
=-\tilde{\chi}(g) u\widetilde{f}+\widetilde{h}u\widetilde{f}
=(\widetilde{h}-\widetilde{h}^{(\neq 0)})u\widetilde{f}
=\widetilde{h}^{(0)}F \/.
\]

Notice that we have $\underline{x}^\alpha\underline{x}^\beta\in V_{d+d', \lambda+\lambda'}$ for any  $\underline{x}^\alpha\in V_{d, \lambda}$ and  $\underline{x}^\beta\in V_{d', \lambda'}$. 
Thus if we  denote by $F_\lambda$ the formal sum of monomials appearing in $F$ of  weight $\lambda$, then $\widetilde{h}^{(0)}F_\lambda$  also has weight $\lambda$. 
Since
$\nu_R$ strictly increase the degree, for any
$\lambda\neq 0$, an argument using the minimal degree of $F_\lambda$ implies that 
$\nu_S(F_\lambda)=0$. Consequently we have 
$\nu_S(F)=0$.

Since $u$ is a unit, the Jacobian ideal $\cJ_F$ is  $\widehat{\fm}_y$-primary and $\{F_{y_0}, \cdots ,F_{y_n}\}$   form a regular sequence. 
Then the syzygy relation $\nu_S(F)=0$ is exact in the Koszul complex, i.e., 
there exists a $(n+1)\times (n+1)$ skew-symmetric formal matrix $K(\underline{x})$
such that 
$S\underline{x}=K(\underline{x})\cdot \nabla_F$.
Taking derivatives and 
evaluating at $0$ we obtain $S=K(0)\cdot \Hess_F(0)$. Consequently we have 
\[
\rk S\leq\rk\Hess_F(0)=\rk\Hess_f(0). \qedhere 
\]
\end{proof}

Now we can prove Theorem~\ref{theo; GSVlowerboundintro} in the case where $n$ is even.
\label{sec; proof}
\begin{theo}
\label{theo; GSVlowerbound}	
We assume that $n$ is even.
If there exists a holomorphic vector field $\tilde{\nu}$ such that $D$, $W$ and $\tilde{\nu}$ satisfy \textup{Convention}~\ref{conv} and that
$\GSV(\tilde{\nu}, D, 0)=1+\tau$ achieves the minimal value, then $D$ has a  weighted homogeneous singularity at $0$.
\end{theo}
\begin{proof}
Let $h=\frac{\tilde{\nu}(f)}{f}$. If $h(0)\neq 0$ then $h\in \cO_{W,0}$ is a unit and we have 
$f=\sum_{i=0}^n \frac{\tilde{\nu}_i}{h}\cdot f_{x_i}\in \cJ_f$. Thus we may assume that $h(0)=0$.

By Proposition~\ref{prop; algebraicGSV}, setting $A:=\frac{\cO_{W,0}}{\cI_{\tilde{\nu}}}$ we have 
\[
1=\dim_\CC A
-\dim_\CC \frac{A}{([h]) }
=\dim_\CC ([h]) \/.
\]
Thus $[h]\neq 0$ is a socle element. 
Then by Proposition~\ref{prop; ranklinearpart} and  Lemma~\ref{lemm; Hessianrank} we have 
\[
 \rk \Hess_f(0)\geq \rk M_{\tilde{\nu}}(0)=n \/.
\]	

If $\rk \Hess_f(0)=n+1$, by the complex Morse lemma $f=y_0^2+y_1^2+\cdots+y_n^2$ for some suitable local coordinates $\{y_0, \cdots ,y_n\}$ and  $f$ is  weighted homogeneous. 
Thus we may   assume $\rk \Hess_f(0)=n$.

Recall the analytic splitting lemma \cite[Theorem~I-2.47]{GLS07}, which   says that there exists  local analytic coordinates $\{y_0, \cdots ,y_n\}$ 
such that   
$f=y_1^2+\cdots+y_{n}^2+g(y_0)$ for some holomorphic function $g(t)\in \fm_{t}^3$. 
Since $D$ has an isolated singularity we see that  $g\not\equiv0$. Thus we may write
$g(t)=t^m u(t)$ for some unit $u(t)$ and some integer $m\geq 3$.

Since $u(0)\neq 0$ we may choose an analytic brach  $v(t)$ of $\sqrt[m]{u(t)}$. 
Up to necessary shrinking    $v(t)$ is a  
 holomorphic germ on $(W, 0)$ 
 such that $v(t)^m=u(t)$.
Then  $\{z:=y_0\cdot v(y_0), y_1, \cdots ,y_n\}$ form a  local analytic coordinates of $W$. Under this coordinates   $f=y_1^2+\cdots+y_{n}^2+ z^m$ is clearly  weighted homogeneous.
\end{proof}

We then obtain the following criterion.
 \begin{coro}
\label{coro; GSVlowerbound}	
The hypersurface germ $D$ is weighted homogeneous at $0$ if and only if there exists a logarithmic vector field germ $\tilde{\nu}\in \Der_{W,0}(-\log D)$ with an isolated singularity at $0$ such that 
\[
\GSV(\tilde{\nu}, D, 0)=1+(-1)^n\tau \/.
\] 
achieves the minimal value.
\end{coro}
\begin{proof}
When $D$ is weighted homogeneous, the Euler vector field always achieves minimal value for the $\GSV$ index. We only need to prove the converse. 

Assuming the existence of $\tilde{\nu}$ such that $\GSV(\tilde{\nu}, D, 0)=1+(-1)^n \tau$. When $n$ is even, $D$ is weighted homogeneous by Theorem~\ref{theo; GSVlowerbound}. When $n$ is odd,
 Proposition~\ref{prop; algebraicGSV} says $
\dim_\CC \frac{\cO_{W,0}}{(f)+\cI_{\tilde{\nu}}}=1$, equivalently we have
\[
(f)+(\tilde{\nu}_0, \cdots , \tilde{\nu}_n)=\fm=(x_0, x_1, \cdots ,x_n) \/.
\]	
Since $f(0)=0$ and $f$ has an isolated singularity at $0$, $f\in \fm^2$ and therefore 	
$(\tilde{\nu}_0, \cdots , \tilde{\nu}_n)=\fm$. Thus $\tilde{\nu}$ is a non-degenerate vector field, and $D$ is weighted homogeneous by Theorem~\ref{theo; non-degenerate}.
\end{proof}

\vskip .2in

\noindent
{\bf Data availability}: No datasets were generated or analyzed during the current study. Data sharing is therefore not applicable to this article.
\vskip .1in

\noindent
{\bf Conflict of Interest}: On behalf of all authors, the corresponding author states that there is no conflict of interest.

\bibliographystyle{plain}
\bibliography{socle}

\end{document}